\documentclass[12pt]{amsart}
\usepackage{amssymb,amsmath,amsthm}
\usepackage{graphicx}
\usepackage{subcaption}
\usepackage{fullpage}
\usepackage{scrextend}
\usepackage{siunitx}
\usepackage{enumerate}
\usepackage{tikz,calc}
\usepackage{tikz-cd}
\usetikzlibrary{automata,positioning}
\usepackage{comment}
\usetikzlibrary{calc}
\usepackage{epsfig,graphicx}
\usepackage{listings}
\usepackage{enumerate}
\usepackage{float}
\usepackage{bbm}
\usepackage[colorlinks=true, pdfstartview=FitH, linkcolor=blue, citecolor=blue, urlcolor=blue]{hyperref}

\newcommand{\Z}{\mathbb{Z}}
\newcommand{\Li}{\operatorname{Li}}
\newcommand{\Mod}[1]{\ (\text{mod}\ #1)}
\renewcommand{\Mod}[1]{{\ifmmode\text{\rm\ (mod~$#1$)}\else\discretionary{}{}{\hbox{ }}\rm(mod~$#1$)\fi}}

\newtheorem{theorem}{Theorem}[section]
\newtheorem{prop}[theorem]{Proposition}
\newtheorem{lemma}[theorem]{Lemma}
\newtheorem{cor}[theorem]{Corollary}

\newtheorem{example}[theorem]{Example}

\theoremstyle{definition} 
\newtheorem{defn}[theorem]{Definition}

\numberwithin{theorem}{section}

\begin{document}

\title{Asymptotics for the number of directions determined by $[n] \times [n]$ in $\mathbb{F}_p^2$ }
\author{Greg Martin}
\address{Department of Mathematics \\ University of British Columbia \\ Room 121, 1984 Mathematics Road \\ Vancouver, BC, Canada V6T 1Z2}
\email{gerg@math.ubc.ca}
\author{Ethan Patrick White}
\address{Department of Mathematics \\ University of British Columbia \\ Room 121, 1984 Mathematics Road \\ Vancouver, BC, Canada V6T 1Z2}
\email{epwhite@math.ubc.ca}
\author{Chi Hoi Yip}
\address{Department of Mathematics \\ University of British Columbia \\ Room 121, 1984 Mathematics Road \\ Vancouver, BC, Canada V6T 1Z2}
\email{kyleyip@math.ubc.ca}
\maketitle

\begin{abstract}
Let $p$ be a prime and $n$ a positive integer such that $\sqrt{\frac p2} + 1 \leq n \leq \sqrt{p}$. For any arithmetic progression $A$ of length $n$ in $\mathbb{F}_p$, we establish an asymptotic formula for the number of directions determined by $A \times A \subset \mathbb{F}_p^2$. The key idea is to reduce the problem to counting the number of solutions to the bilinear Diophantine equation $ad+bc=p$ in variables $1\le a,b,c,d\le n$; our asymptotic formula for the number of solutions is of independent interest.
\end{abstract}

\section{Introduction}

\subsection{The number of directions determined by a set of ordered pairs}

Let $F$ be a field, and let $U \subset F^2$ be a finite set of ordered pairs. The set of \emph{directions determined} by~$U$ is defined to be
\begin{equation} \label{directions def}
\mathcal{D}_U =  \left\{ \frac{b-d}{a-c} \colon (a,b), (c,d) \in U,\, (a,b) \neq (c,d) \right\}
\end{equation}
considered as a subset of $F \cup \{ \infty\}$, where $\infty$ is the vertical direction resulting from $a=c$. The theory of directions is well studied, particularly when $F=\mathbb{F}_p$ is a finite field---see for example~\cite{BB, Sz}. One of the most important results in the subject is the following lower bound on the cardinality of $\mathcal{D}_U$, which was proved by R\'edei~\cite{LR} in the case $|U|=p$ and later extended by Sz\H{o}nyi~\cite[Theorem 5.2]{Sz} to any $|U| \leq p$.

\begin{theorem}[Sz\H{o}nyi]\label{Sz}
Let $p$ be a prime, and let $U \subset \mathbb{F}_p^2$ with $1<|U|\leq p$. Then either $U$ is contained in a line, or $U$ determines at least $\frac{|U|+3}{2}$ directions.
\end{theorem}

Di Benedetto, Solymosi, and the second author~\cite[Theorem 1]{DSW} improved Theorem~\ref{Sz} when $U$ has a Cartesian product structure.

\begin{theorem}[Di Benedetto/Solymosi/White]\label{dswT} 
Let $p$ be a prime, and let $A,B\subset \mathbb{F}_p$ be sets each of size at least~$2$ such that $|A||B| < p$. Then the set of points $A\times B\subset \mathbb{F}_p^2$ determines at least $|A||B| - \min\{|A|,|B|\} + 2$ directions.
\end{theorem}

We remark that the set of directions determined by $A \times A$ in Theorem~\ref{dswT} is the set $(A-A)/(A-A)$. Estimating the size of $(A-A)/(A-A)$ is often a critical step in sum-product and character-sum results over finite fields---see for example~\cite{MPR,SS}.

For positive integers $n$, define $[n] = \{1,2,\ldots,n\}$. The authors of~\cite{DSW} observed that Theorem~\ref{dswT} is tight for long rectangles of the form $[3] \times [2n-1]$, and speculated that Theorem~\ref{dswT} might be improved for Cartesian products of the form $A \times A$. In this work we show that this is indeed the case for $A = [n]$ by determining an asymptotic formula for the number of directions determined by $[n]^2 \subset \mathbb{F}_p^2$, for all primes~$p$.

The statement of our main theorem involves the continuous function
\begin{equation}\label{D def}
D(\lambda) = \begin{cases}
\frac{12}{\pi^2} \lambda^2, & \lambda \in [0,\frac1{\sqrt2}], \\
 \frac{6}{\pi^2} \big( 2\Li_2(\lambda^2) + \log^2(\lambda^2)  - 2(1-\lambda^2)\log(\lambda^{-2}-1) + 2(1-\lambda^2) \big) - 1, & \lambda \in (\frac1{\sqrt2},1), \\
1, & \lambda \geq 1 ,
\end{cases}
\end{equation}
where
\begin{equation}\label{Li}
\Li_2(x) = -\int _{0}^{x}\frac{\log(1-t)}{t}\,dt
\end{equation}
is the {\em dilogarithm function}.
Figure~\ref{ffig} shows the graph of~$D(\lambda)$ as a gold, turquoise, and red (solid and dashed) curve, where each color represents a piece of the piecewise defined function. The purple (dotted) curve is the graph of $y=\lambda^2$, whose significance we will mention momentarily.

\begin{figure}[h!]
  \centering
    \includegraphics[width=0.6\textwidth]{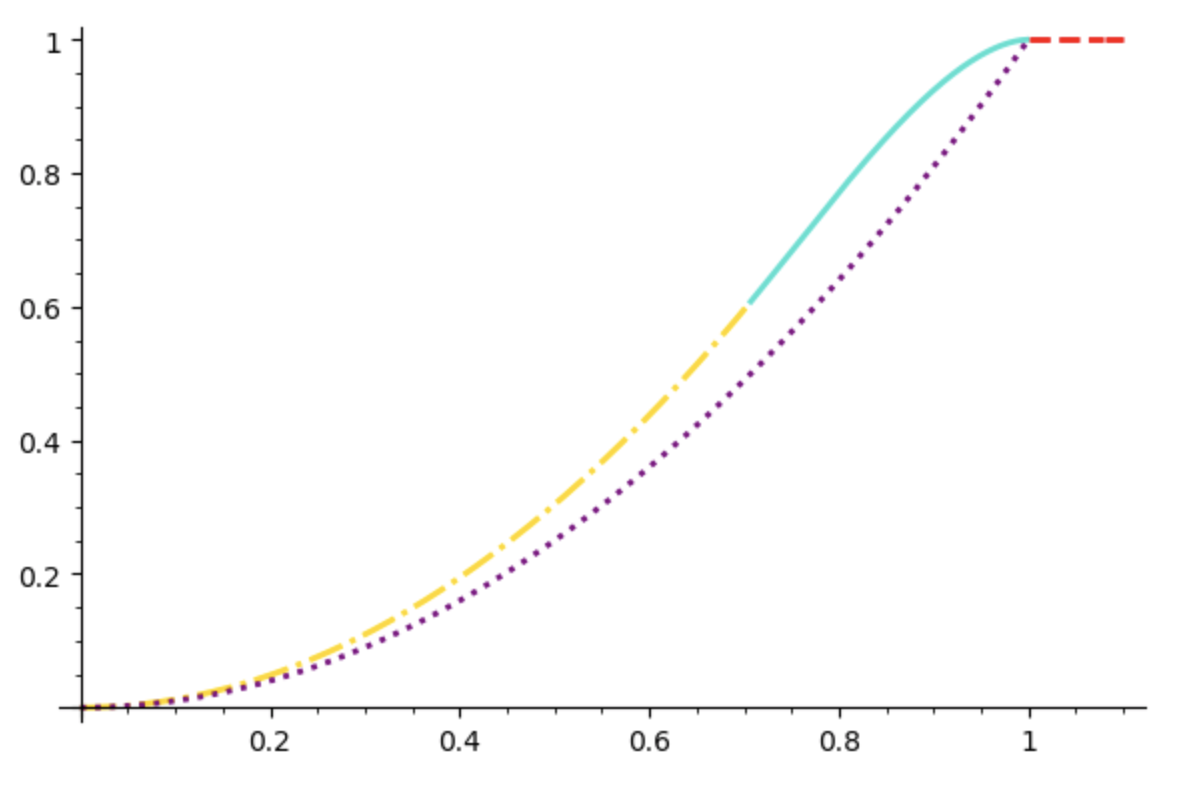}
    \caption{Comparison of $D(\lambda)$ (upper curve) and~$\lambda^2$, the bound from~\cite{DSW} (lower curve). The three textures of the upper curve indicate the three cases of the definition~\eqref{D def} of~$D(\lambda)$.}
    \label{ffig}
\end{figure}

We may now state our main result, an asymptotic formula for the number of directions determined by $[n]^2$, or equivalently by $A^2$ for any arithmetic progression $A$ of length $n$ in $\mathbb{F}_p$ (since all such sets are linearly equivalent and thus yield the same set of directions).

\begin{theorem}\label{maind}
The number of directions determined by $[n]^2 \subset \mathbb{F}_p^2$ is 
\[ D\bigg(\frac n{\sqrt p}\bigg)p + O\big(p^{3/4}(\log p)^{2^{3/2}+1}\big),\]
where the function $D(\lambda)$ is defined in equation~\eqref{D def}.
\end{theorem}

It is a simple consequence of the pigeonhole principle that any set $U \subset \mathbb{F}_p^2$ with $|U|>p$ will determine all $p+1$ directions (since in $\mathbb{F}_p^2$ there are exactly $p$ lines of any particular slope). Therefore, estimating the number of directions determined by  $[n]^2 \subset \mathbb{F}_p^2$ is interesting only when $n \leq \sqrt{p}$. On the other hand, if $n$ is sufficiently small relative to $p$, it is easy to see that the number of directions determined by $[n]^2$ is the same in $\mathbb{Q}$ and $\mathbb{F}_p$; Lemma~\ref{toosmall} gives a precise statement to this effect when $n\le \sqrt{\frac p2}$. The interesting range $n \in (\sqrt{\frac p2},\sqrt{p})$, corresponding to the case $\lambda \in (\frac1{\sqrt2},1)$ in the definition of~$D(\lambda)$, is the range in which Theorem~\ref{maind} is nontrivial and new. In particular, previously there was no nontrivial upper bound, while Theorem~\ref{dswT} applied to the case $A=B=[n]$ gave the best known lower bound (depicted by the purple curve in Figure~\ref{ffig}); note that our result actually gives an asymptotic formula in this interesting range.

We remark that the dilogarithm function~$\Li_2$ appears in many different contexts in number theory. In particular, Cilleruelo and Guijarro-Ord\'o\~nez~\cite{CG} showed that the typical size of the ratio set $A/A$ for a random set $A \subset [n]$ also involves the dilogarithm function.

The main ingredient of the proof of Theorem~\ref{maind} is Theorem~\ref{maint} below, which is a purely number-theoretical statement giving an asymptotic formula for the number of solutions to the bilinear Diophantine equation $ad+bc=p$ with the variables in~$[n]$. Our proof could be adapted to obtain an asymptotic formula for the number of directions determined by $[m] \times [n] \subset \mathbb{F}_p^2$, although the details would be more complicated.

Since the difference between the number of directions determined by $[n]^2$ and by $[n+1]^2$ can be greater than~$n$ (when~$n$ is prime, for example), an error term of at least $O(\sqrt{p})$ is unavoidable in Theorem~\ref{maind}. We conjecture that the true size of the error term is $\sqrt{p}$ up to logarithmic factors.
Incidentally, note that the difference between the number of directions determined by $[n]^2$ and by $[n+1]^2$ is also $O(n)$, which gives another way to see that the function $D(\lambda)$ appearing in the asymptotic formula in Theorem~\ref{maind} is Lipschitz continuous.

We now give several consequences of Theorem~\ref{maind}. 
The following theorem, due to Solymosi~\cite[Theorem 4]{TV}, provides an extension of the classical Thue--Vinogradov lemma from elementary number theory. We use the notation $\mathbb{F}_p^* = \mathbb{F}_p \setminus \{0\}$.

\begin{theorem}[Solymosi]\label{TV}
Let $p$ be a prime. For any $\alpha,\beta\in \mathbb{N}$ satisfying $\alpha(\beta+1)\leq p-1$, there are at least
$\alpha(\beta+1)$ distinct elements
$a\in\mathbb{F}_p^*$ for which there exist $x \in [\alpha]$ and $y \in [\beta]$ such that $ax \equiv \pm y \Mod{p}$.
\end{theorem}

Both the statement and the proof of Theorem~\ref{TV} share some similarity with Theorem~\ref{dswT}; in fact, Theorem~\ref{TV} can be viewed as a lower bound for the number of directions determined by $[\alpha] \times [\beta]$. Therefore we see that Theorem~\ref{maind} immediately improves Theorem~\ref{TV} in the case $\alpha=\beta<\sqrt p$; see also~\cite[Remark 6]{TV} for a related discussion when $\alpha=\beta<\sqrt{\frac p2}$. The lower curve in Figure~\ref{ffig} represents the lower bound in Theorem~\ref{TV}, while the upper curve represents our asymptotic formula, which can be rephrased as follows.

\begin{cor}\label{TVcor}
Let $p$ be a prime. For any positive integer~$n$ satisfying $n < \sqrt{p}$, there are
$D(\frac n{\sqrt p})p+O\big(p^{3/4}(\log p)^{2^{3/2}+1}\big)$ distinct elements
$a\in\mathbb{F}_p^*$ for which there exist $x,y \in [n]$ 
such that $ax \equiv \pm y \Mod{p}$.
\end{cor}

It is possible to generalize Theorem~\ref{Sz} and Theorem~\ref{dswT} to more general finite fields, although the methods become more technical; see for example the survey paper by Sz\H{o}nyi~\cite{Sz} and recent papers by Dona~\cite{DD} and the third author~\cite{Yip}. On the other hand, it is straightforward to generalize our Theorem~\ref{maind} to an arbitrary finite field, since any arithemtic progression of length~$n$ in $\mathbb{F}_{p^k}$ is still linearly equivalent to~$[n]$. Moreover, we can replace $[n]$ by sets $A$ such that $A-A$ contains a long arithmetic progression, since the size of the direction set $\mathcal{D}_{A\times A} = (A-A)/(A-A)$ is invariant under any affine transformation of~$A$. Recall that a {\em homogeneous arithmetic progression} is an arithmetic progression whose first term is equal to its common difference. Theorem~\ref{maind} immediately implies the following lower bound:

\begin{cor} \label{dcor1}
Let $F$ be a field with characteristic $p$. Let $A \subset F$ be a set such that $A-A$ contains a homogeneous arithmetic progression of length $n$. Then the number of directions determined by $A \times A \subset F^2$ is at least 
\[ D\bigg(\frac n{\sqrt p}\bigg)p + O\big(p^{3/4}(\log p)^{2^{3/2}+1}\big).\]
\end{cor}

We digress slightly to give an example where Corollary~\ref{dcor1} is much stronger than the earlier theorems. Since we will be content with noting the relevant orders of magnitude rather than the leading constants, we note that the conclusions in this example follow already from Theorem~\ref{dswT} once one notes that it suffices to consider homogeneous arithmetic progressions in the difference set $A-A$.

\begin{example} \rm 
The Stanley sequence $S \subset \mathbb{Z}$ consists of all nonnegative integers whose base-$3$ representation contains only the digits~$0$ and~$1$. It was introduced in~\cite{OS} as an example of a set containing no arithmetic progressions of length~$3$; for our purposes, however, the relevant property is that its finite truncations $S_n = S \cap [0,n]$ are small yet their difference sets contain long homogeneous arithmetic progressions. In particular, $n^\eta \ll |S_n| \ll n^\eta$ where $\eta=\frac{\log2}{\log3} \approx 0.631$; but it is also easy to show that $[1,\frac n3] \subset S_n-S_n$.

Fix a prime $p\ge n$ and consider $S_n \times S_n$ as a subset of $\mathbb{F}_p^2$.
A direct application of Theorem~\ref{Sz} shows only that the number of directions determined by $S_n \times S_n$ is $\gg \min\{|S_n|^2,p\}$. However, Corollary~\ref{dcor1} shows that the number of directions determined by $S_n \times S_n$ is $\gg \min\{n^2,p\} \gg \min\{|S_n|^{2/\eta},p\}$, where $\frac2\eta \approx 3.17$. (We also note that no better upper bound for the number of directions determined by $S_n\times S_n$ is known other than the trivial $\min\{|S_n|^4,p\}$.)
\end{example}

Next we consider the case where $A$ is ``close to" an arithmetic progression, in which case we expect the {\em doubling constant} $|A-A|/|A|$ to be small. The structure of sets with small doubling constant has been widely studied since Freiman's seminal work~\cite{Freiman}. Freiman's ``$3k-4$ theorem'' states that if $A$ is a finite set of integers satisfying $|A+A| \leq 3|A|-4$, then $A$ must be contained in a short arithmetic progression; his celebrated ``$2.4$ theorem''~\cite[Theorem 2.1]{Freiman} is a similar statement in the finite field setting. Recently, Lev and Shkredov~\cite{LS} refined Freiman's work and showed the following ``$2.6$ theorem'' in terms of $A-A$.

\begin{theorem}[Lev/Shkredov]\label{doubling}
Let $p$ be a prime. If $A \subset \mathbb{F}_p$ such that $|A|<0.0045p$, and $|A-A| \leq 2.6 |A|-3$, then $A$ is contained in an arithmetic progression $P$ with at most $|A-A|-|A|+1$ terms.
\end{theorem}

Note that if $A$ is contained in an arithmetic progression $P$, then we can use Theorem~\ref{maind} to give an asymptotic formula for the number of directions determined by $P \times P$, and hence an upper bound for the number of directions determined by $A \times A$. This observation, together with Theorem~\ref{doubling}, immediately imply the following corollary.

\begin{cor} \label{dcor2}
Let $p$ be a prime. Let $A \subset \mathbb{F}_p$ satisfy $|A| \leq \sqrt{p}$ and $|A-A| \leq 2.6 |A|-3$. Then the number of directions determined by $A \times A \subset \mathbb{F}_p^2$ is at most 
\[ D\bigg(\frac{|A-A|-|A|+1}{\sqrt{p}}\bigg) p + O\big(p^{3/4}(\log p)^{2^{3/2}+1}\big),\]
where the function $D(\lambda)$ is defined in equation~\eqref{D def}.
\end{cor}

\subsection{The number of solutions to $ad+bc=p$}

In Section \ref{reduction}, we will reduce the problem of counting the number of directions to estimating the number of solutions to the Diophantine equation $ad+bc=p$. For convenience, we introduce the following notation.

\begin{defn}
Let $N(p,n)$ denote the number of solutions $(a,b,c,d) \in [n]^4$ to the equation $ad+bc=p$.
\end{defn}

Since $ad+bc \le 2n^2$ when $(a,b,c,d) \in [n]^4$, we see that $N(p,n)=0$ trivially when $n<\sqrt{\frac p2}$. Our second main theorem gives an asymptotic formula for $N(p,n)$ in the same ``interesting'' range $n \in (\sqrt{\frac p2},\sqrt{p})$ as in the previous section.
Throughout this paper, it will be convenient for us to define the positive parameter $\lambda$ which will always have the following relationship with $p$ and $n$:
\begin{equation} \label{lambda}
n = \lambda \sqrt{p}.
\end{equation}

\begin{theorem}\label{maint}
Let $p$ be a prime, let $\frac1{\sqrt2}<\lambda<1$, and set $n=\lambda\sqrt{p}$. The number of solutions $(a,b,c,d) \in [n]^4$ to the equation $ad+bc=p$ is
\begin{equation}\label{asymptotic}
N(p,n) = \bigg(\frac{12}{\pi^2} \lambda^2 -D(\lambda)\bigg) p +O\big(p^{3/4}(\log p)^{2^{3/2}+1}\big),
\end{equation}
where the function $D(\lambda)$ is defined in equation~\eqref{D def}, and the implied constant in the error term is absolute.
\end{theorem}

\noindent
We will see in Section~\ref{reduction} that Theorem~\ref{maint} (together with Lemma~\ref{connect}) implies Theorem~\ref{maind}. Therefore our main task in this paper is to prove Theorem~\ref{maint}.

For the rest of this section, we explore some interesting consequences of Theorem~\ref{maint}, as well as revisiting known upper bounds and lower bounds on $N(p,n)$.

Note that when $n<\sqrt{p}$, we have $0<ad+bc<2p$ whenever $a,b,c,d \in [n]$; in this range, therefore, the equation $ad+bc=p$ is equivalent to the congruence $-ad \equiv bc \Mod p$.
The congruence $ab \equiv cd \Mod p$ is well studied (see for example~\cite{ACZ, CS}), and
similar bilinear congruences have been examined by many mathematicians.
The standard way to estimate the number of solutions to such bilinear congruences is to estimate fourth moments of character sums.

For any integer $u$ not divisible by~$p$, let $N_u(p,n)$ denote the the number of solutions $(a,b,c,d) \in [n]^4$ to the congruence $uab \equiv cd \Mod p$, where $a,b,c,d \in [n]$.  From the standard orthogonality relation
\[
\sum_{\chi\Mod p} \chi(a)=
\begin{cases}
p-1, \quad & \text{if } a \equiv 1 \Mod p,\\
0, \quad &\text{otherwise},
\end{cases}
\]
where the sum runs over all Dirichlet characters modulo~$p$ (refer to~\cite[Chapter 4]{MV} for example), it follows that
\begin{equation} \label{oddeven}
N_u(p,n)=\frac{1}{p-1} \sum_{\chi\Mod p} \sum_{1 \leq a,b,c,d \leq n} \chi(uabc^{-1}d^{-1})= \frac{1}{p-1} \sum_{\chi\Mod p} \chi(u) \bigg|\sum_{m=1}^{n} \chi(m)\bigg|^4.
\end{equation}
When $\chi=\chi_0$ is the principal character, each $\chi(m)$ in the inner sum equals~$1$, and therefore the contribution to the right-hand side from $\chi=\chi_0$ is exactly $\frac{n^4}{p-1}$; on the other hand, Ayyad, Cochrane, and Zheng~\cite[Theorem~2]{ACZ} showed that
\begin{equation} \label{fourth}
\frac{1}{p-1} \sum_{\chi \neq \chi_0} \bigg|\sum_{m=1}^{n}  \chi(m)\bigg|^4 \ll \begin{cases}
n^2 \log^2 p, &\text{for all } n, \\
n^2 \log p, &\text{if } n \ll \sqrt{p\log p}.
\end{cases}
\end{equation}
It follows that for any $1\le n\le p-1$ and $1\le u\le p-1$,
\begin{equation}\label{trivialb}
N_u(p,n)=\frac{n^4}{p-1}+
\begin{cases}
O(n^2 \log^2 p), &\text{for all } n, \\
O(n^2 \log p), &\text{if } n \ll \sqrt{p\log p}
\end{cases}
\end{equation}
(as remarked at the end of~\cite{ACZ}; see also~\cite[Lemma 5]{CS} for a related discussion).
Although this character sum approach succeeds in obtaining the asymptotic formula~\eqref{trivialb} when $n$ grows faster than $\sqrt{p}\log p$, over short intervals the estimates are poorer. In particular, in the range $n <\sqrt{p}$, for which $N(p,n) = N_{-1}(p,n)$, equation~\eqref{trivialb} simply states that $N(p,n) \ll p \log p$, which is a poor estimate compared to Theorem~\ref{maint}.

We have just seen that the right-hand side of equation~\eqref{fourth} is dominated by the contribution of the principal character~$\chi_0$ when $n$ grows faster than $\sqrt{p}\log p$. However, it turns out that the situation is drastically different when $n<\sqrt p$. Note that it follows immediately from equation~\eqref{oddeven} (and the fact that $\chi(1)=1$ always) that
\begin{align}\label{Nprime}
N_1(p,n) + N_{-1}(p,n) &= \frac2{p-1} \sum_{\substack{\chi\Mod p \\ \chi(-1) = 1}} \bigg|\sum_{m=1}^{n} \chi(m)\bigg|^4 , \\
N_1(p,n) - N_{-1}(p,n) &= \frac2{p-1} \sum_{\substack{\chi\Mod p \\ \chi(-1) = -1}} \bigg|\sum_{m=1}^{n} \chi(m)\bigg|^4. \nonumber
\end{align}
When $n<\sqrt p$, we have already seen that $N_{-1}(p,n)=N(p,n)$ and thus $N_{-1}(p,n) \ll n^2$ by Theorem~\ref{maint}. On the other hand, Ayyad, Cochrane, and Zheng~\cite[Theorem 3]{ACZ} obtained an asymptotic formula for the number of solutions $(a,b,c,d) \in [n]^4$ to $ad=bc$. In particular, when $n <\sqrt{p}$ the equation $ad=bc$ is similarly equivalent to the congruence $ad\equiv bc\pmod p$, and their result becomes the asymptotic formula
\[
N_1(p,n)=\frac{12}{\pi^2}n^2 \log n+ O(n^2) \quad\text{when }n<\sqrt p.
\]
Consequently, equation~\eqref{Nprime} implies the following result:

\begin{cor} \label{odd=even}
Let $p$ be a prime. If $n < \sqrt{p}$, then 
\begin{align*}
\frac{1}{p-1} \sum_{\chi(-1)=1} \bigg|\sum_{m=1}^{n} \chi(m)\bigg|^4 &= \frac{6}{\pi^2} n^2\log n+O(n^2) , \\
\frac{1}{p-1} \sum_{\chi(-1)=-1} \bigg|\sum_{m=1}^{n} \chi(m)\bigg|^4 &= \frac{6}{\pi^2} n^2\log n+O(n^2).
\end{align*}
\end{cor}

\noindent
In other words, when examining the fourth moment of character sums modulo~$p$ (that is, the right-hand side of equation~\eqref{oddeven} when $u=1$), the contribution from odd characters is asymptotically equal to the contribution from even characters when $n<\sqrt p$. This is a stark contrast to the dominance of the principal character when~$n$ is only a bit larger than $\sqrt{p}\log p$, which suggests that it would be interesting to study both sides of equation~\eqref{oddeven} as~$n$ transitions between these two quite close orders of magnitude.

Returning to the number of solutions $N(p,n)$ itself, it seems nontrivial to show from first principles even that $N(p,n)\geq 1$ (that is, that there exists $(a,b,c,d) \in [n]^4$ with $ad+bc=p$) when $n<\sqrt{p}$. One may try to express $N(p,n)$ as a convolution by estimating the number of points of the modular hyperbola $\{(a,b) \in [n]^2\colon ab \equiv x \Mod p\}$
for each $x \in [p]$---see for example Shparlinski's survey paper~\cite{Sh}. Indeed, Hart and Iosevich~\cite{HI} showed that if $A \subset \mathbb{F}_p$ satisfies $|A| \geq p^{3/4}$, then $\mathbb{F}_p^* \subset AA+AA$, where 
$AA=\{ab\colon a,b \in A\}$.
Shparlinski~\cite{Sh0} remarked that the Hart/Iosevich proof could be easily extended to show for any $x \in \mathbb{F}_p^*$ and any $A,B,C, D \subset \mathbb{F}_p^*$, the number of solutions $(a,b,c,d)\in A\times B\times C\times D$ to $ad+bc=x$ is
\begin{equation} \label{ABCD}
\frac{|A||B||C||D|}{p-1} +O\big(\sqrt{p|A||B||C||D|}\big),
\end{equation}
which gives an asymptotic formula when $|A||B||C||D|$ grows faster than $p^3$. We remark that the assumption $x \neq 0$ is necessary: if $p \equiv 1 \Mod 4$, $A=B=C$ is the set of quadratic residues modulo $p$, and $D$ is the set of quadratic non-residues modulo $p$, then $|A||B||C||D| \gg p^4$ yet $0 \notin AD+BC$. 

These known results strongly suggest that it is important to distinguish the congruence $ad+bc \equiv 0 \Mod p$ from the congruences $ad+bc \equiv u \Mod p$ where $p \nmid u$; it seems more difficult to study the first congruence than the latter ones. Indeed, Ayyad and Cochrane~\cite[Theorem 2]{AC} showed that the congruence lattice modulo $p$ is well-distributed as long as $ad+bc \equiv 0 \Mod p$ has a solution in a prescribed region; more precisely, they proved the following result:

\begin{theorem}[Ayyad/Cochrane] \label{AC}
Let $a, b, m$ be integers with $m\ge1$ and $\gcd(a, b, m)=1$, and suppose that the congruence $ax +b y \equiv 0 \Mod m$ has a solution $(x_{0}, y_{0}) \in \mathcal{R}_{m}$ with $\operatorname{gcd}\left(x_{0}, y_{0}\right)=1$, where
\begin{equation} \label{Rm}
\mathcal{R}_{m}=\{(x, y) \in \Z^2\colon 0 \leq|x| \leq \sqrt{m},\, 0 \leq|y| \leq \sqrt{m},\, |x|+2|y| \geq \sqrt{m},\, 2|x|+|y| \geq \sqrt{m}\}.
\end{equation}
Then for any integer $c$, the linear congruence
 $a x+b y \equiv c \Mod m$
has a nonzero solution with $|x| \leq \sqrt{m}$ and $|y| \leq \sqrt{m}$.
\end{theorem}

We show at the end of Section \ref{reduction} that Theorem~\ref{maint} and Theorem~\ref{AC} imply the following corollary, which is well beyond the reach of equation~\eqref{ABCD}:

\begin{cor}\label{additive}
There are at least $(\frac{12}{\pi^2}-1)p+O\big(p^{3/4}(\log p)^{2^{3/2}+1}\big)$ ordered pairs $(a,b) \in \Z^2 \cap [1, \sqrt{p}]^2 $ such that for any integer $c$, the linear congruence
 $a x+b y \equiv c \Mod p$
has a nonzero solution with $|x| \leq \sqrt{p}$ and $|y| \leq \sqrt{p}$.
\end{cor}

In other words, for a positive proportion of pairs of integers $1\le a,b \leq \sqrt{p}$, every congruence of the form $ax+by \equiv c \Mod{p}$ with $c\in\mathbb{Z}$ admits a small solution where $|x|,|y| \leq \sqrt{p}$.

\section{Reduction to $ad + bc = p$}\label{reduction}

The main objective of this section is to transfer the problem of estimating the number of directions determined by $[n]^2 \subset \mathbb{F}_p^2$ to estimating the number of solutions $(a,b,c,d) \in [n]^4$ to the equation $ad + bc = p$.
In particular, we show that Theorem~\ref{maint} implies Theorem~\ref{maind}.
At the end of this section, we also show that Theorem~\ref{maint} and Theorem~\ref{AC} imply Corollary~\ref{additive}.

As a variant of the notation~\eqref{directions def}, for any field~$F$ and any positive integer~$n$ we let
\[ \mathcal{D}_n(F) = \left\{ \frac{a}{b} \in F \colon -n+1 \leq a,b \leq n-1,\, (a,b)\ne(0,0) \right\}
\]
denote the set of directions determined by $[n]^2$ over the field~$F$. (If $F$ has characteristic~$p$ then we add the restriction $n\le p$.)
The size of $\mathcal{D}_n(F)$ will depend on the characteristic of the underlying field: for example, the map from $\mathcal{D}_n(\mathbb{Q})$ to $\mathcal{D}_n(\mathbb{F}_p)$ induced by the natural quotient map from $\mathbb{Z}$ to $\mathbb{F}_p$ is clearly surjective, so that $|\mathcal{D}_n(\mathbb{F}_p)| \le |\mathcal{D}_n(\mathbb{Q})|$, but in general is not injective. As suggested in the introduction, however, this map is injective when~$p$ is large compared to~$n$:

\begin{lemma}\label{toosmall} If $n < \sqrt{\frac p2}$, then $|\mathcal{D}_n(\mathbb{F}_p)|=|\mathcal{D}_n(\mathbb{Q})|$.

\end{lemma} 
\begin{proof} Suppose that there are fewer directions determined over $\mathbb{F}_p$ than over~$\mathbb{Q}$. Then there must exist $a,b,c,d \in [-n+1,n-1]$ and a nonzero integer $k$ such that $ad-bc = kp$. But $|kp|\ge p$, while the triangle inequality implies $|ad-bc| < 2(n-1)^2 < p$ by the assumption $n < \sqrt{\frac p2}$; this contradiction establishes the lemma.
\end{proof}

We split the directions, other than~$0$ and~$\infty$, determined by $[n]^2$ over a field $F$ into ``positive directions'' and ``negative directions'', defining
\[\mathcal{D}^+_{n}(F) =   \left\{ \frac{a}{b} \in F \colon 1 \leq a,b \leq n-1 \right\} , \quad \mathcal{D}^-_{n}(F) =   \left\{ -\frac{a}{b} \in F \colon 1 \leq a,b \leq n-1 \right\}.\]
Note that in $\mathbb{F}_p$ these sets can overlap, though in $\mathbb{Q}$ they are obviously disjoint. A slight modification of the proof of Lemma~\ref{toosmall} shows that the number of positive directions is the same over $\mathbb{F}_p$ as over $\mathbb{Q}$, even when~$n$ is large enough to be within the interesting range $n \in (\sqrt{\frac p2},\sqrt{p})$, and similarly for the number of negative directions.

\begin{lemma}\label{distinct} If $n < \sqrt{p}$, then $|\mathcal{D}^+_{n} (\mathbb{F}_p)| = |\mathcal{D}^+_{n} (\mathbb{Q})|$ and $|\mathcal{D}^-_{n} (\mathbb{F}_p)| = |\mathcal{D}^-_{n} (\mathbb{Q})|$.
\end{lemma}

\begin{proof} Suppose that there are fewer positive directions determined over $\mathbb{F}_p$ than over~$\mathbb{Q}$.
Then there must exist $a,b,c,d \in [n-1]$ and a nonzero integer $k$ such that $ad-bc = kp$. Without loss of generality $k \geq 1$, and so $ad > p$; but by assumption $ad \leq (n-1)^2 < p$, a contradiction. The same argument applies to negative directions.
\end{proof}

Estimating the number of directions determined by $[n]^2 \subset \mathbb{Q}^2$, or equivalently, estimating the number of lattice points in $[n]^2$ that are visible from the origin, is a well-known elementary exercise using M\"{o}bius inversion. Because M\"{o}bius inversion will be a crucial tool for us, we now recall some properties of the M\"{o}bius $\mu$ function, starting with its characteristic property
\begin{equation}\label{key identity}
\sum_{d|n}\mu(d)=\begin{cases}
1, & \text{if }n=1,\\
0, & \text{otherwise.}
\end{cases}
\end{equation}
We also use the asymptotic formula
\begin{align}\label{musum2}
\sum_{d\le x} \frac{\mu(d)}{d^2} = \frac6{\pi^2} + O\bigg( \frac1x \bigg)
\end{align}
for $x\ge2$ (see for example the proof of~\cite[Theorem 2.1]{MV}), as well as the asymptotic formula for the harmonic numbers
\begin{equation} \label{harmonic}
\sum_{d\le x} \frac1d = \log x + O(1).
\end{equation}
For the sake of completeness and to foreshadow our later methods, we give a proof of an estimate for the number of directions determined by $[n]^2 \subset \mathbb{Q}^2$. 

\begin{lemma}\label{Qcount} The number of positive directions determined by $[n]^2 \subset \mathbb{Q}^2$ is $\frac{6}{\pi^2}n^2 + O(n\log n)$, and the same is true for the number of negative directions.
\end{lemma}

\begin{proof} It suffices to consider $|\mathcal{D}^+_{n}(\mathbb{Q})|$ since $\mathcal{D}^-_{n}(\mathbb{Q})=-\mathcal{D}^+_{n}(\mathbb{Q})$. By the characteristic property~\eqref{key identity} of the M\"obius function,
\begin{align*}
|\mathcal{D}^+_{n}(\mathbb{Q})| & = \left| \left\{ \frac{a}{b} \in \mathbb{Q} \colon 1 \leq a,b \leq n-1 \right\} \right| \\
& = \sum_{\substack{1 \leq a,b \leq n-1 \\ \gcd(a,b) = 1}} 1 
 = \sum_{1 \leq a,b \leq n-1 } \sum_{d\mid \gcd(a,b)} \mu(d) \\
& = \sum_{1\le d\le n-1} \mu(d) \sum_{\substack{1 \leq a,b \leq n-1 \\ d\mid a,\, d\mid b}} 1 \\
& = \sum_{d =1 }^{n-1} \mu(d) \left\lfloor \frac{n-1}{d} \right\rfloor^2 
 = \sum_{d =1 }^{n-1} \mu(d) \left(\frac{n}{d}+O(1)\right)^2 \\
& = n^2 \sum_{d =1 }^{n-1} \frac{\mu(d)}{d^2}+ O\left(n\sum_{d =1 }^{n-1} \frac{1}{d} \right) + O\left( \sum_{d=1}^{n-1} 1\right) 
= \frac{6}{\pi^2}n^2 + O(n\log n)
\end{align*}
by equations~\eqref{musum2} and~\eqref{harmonic}.
\end{proof}

In the following lemma we see the significant connection between the size of $\mathcal{D}_n(\mathbb{F}_p)$ and $N(p,n-1)$. In particular, it is immediate that Lemma~\ref{connect} and Theorem~\ref{maint} together imply  Theorem~\ref{maind}.

\begin{lemma}\label{connect}
Let $p$ be a prime and $n$ a positive integer with $n<\sqrt{p}$. Then the number of directions determined by $[n]^2 \subset \mathbb{F}_p^2$ is 
\[ \frac{12}{\pi^2}n^2 - N(p,n-1) + O(n \log n).\]
\end{lemma}
\begin{proof}
Note that 
\begin{align}\label{dbreak}
|\mathcal{D}_n(\mathbb{F}_p)| &= |\mathcal{D}_n^+(\mathbb{F}_p)| + |\mathcal{D}_n^-(\mathbb{F}_p)| - |\mathcal{D}_n^+(\mathbb{F}_p)\cap \mathcal{D}_n^-(\mathbb{F}_p)|+2 \\
&= |\mathcal{D}_n^+(\mathbb{Q})| + |\mathcal{D}_n^-(\mathbb{Q})| - |\mathcal{D}_n^+(\mathbb{F}_p)\cap \mathcal{D}_n^-(\mathbb{F}_p)|+2 \notag \\
&= \frac{12}{\pi^2}n^2 + O(n \log n) - |\mathcal{D}_n^+(\mathbb{F}_p)\cap \mathcal{D}_n^-(\mathbb{F}_p)| \notag
\end{align}
by Lemmas~\ref{distinct} and~\ref{Qcount}
(where the~$2$ counts the directions~$0$ and~$\infty$). Any element $x \in \mathcal{D}_n^+(\mathbb{F}_p)\cap \mathcal{D}_n^-(\mathbb{F}_p)$ must be simultaneously of the form $x = \frac ab \in \mathbb{F}_p$ and $x = -\frac cd \in \mathbb{F}_p$ where $a,b,c,d \in [n-1]$, which implies that $ad+bc \equiv bdx-bdx=0 \Mod{p}$. Since $n< \sqrt{p}$ and therefore $0<ad+bc<2(n-1)^2<2p$, we conclude that $ad+bc = p$. Furthermore, each solution to $ad+bc = p$ with $a,b,c,d \in [n-1]$ corresponds to the unique element $\frac ab = -\frac cd \in \mathcal{D}_n^+(\mathbb{F}_p)\cap \mathcal{D}_n^-(\mathbb{F}_p)$ (to verify the uniqueness, it helps to note that $ad+bc=p$ implies that $\gcd(a,b)=\gcd(c,d)=1$). Therefore $|\mathcal{D}_n^+(\mathbb{F}_p)\cap \mathcal{D}_n^-(\mathbb{F}_p)| = N(p,n-1)$, which completes the proof.
\end{proof}

The remainder of this paper is devoted to proving Theorem~\ref{maint}, so that we are interested in the range $n\in (\sqrt{\frac p2},\sqrt{p})$. Consistent with equation~\eqref{lambda}, we define $\lambda = \frac{\sqrt{p}}n$, a convention that will hold throughout even when not explicitly mentioned (as will the assumption that $p$ is a prime). Anytime we use $O(\cdot)$ or $\ll$ notation, the implied constants are absolute unless dependence on particular parameters is explicitly indicated by subscripts; in particular, these implied constants are uniform in~$\lambda$.

Note that if $(a,b,c,d) \in [n]^4$ is a solution to $ad+bc = p$, then since $c \leq \lambda\sqrt{p}$ and $d \leq \lambda\sqrt{p}$ we must have $a+b \geq \frac{\sqrt{p}}\lambda$; furthermore, we have $\gcd(a,b) = 1$ since $p$ is a prime. By symmetry, we also have $c+d \geq \frac{\sqrt{p}}\lambda$ and $\gcd(c,d) = 1$. It is therefore useful to define the set of visible lattice points in a triangular region,
\begin{equation}\label{T}
 T = T(\lambda,p) = \Big\{(a,b) \in \mathbb{Z}^2 \colon 1\le a,b \le \lambda \sqrt{p}, \, a+b \geq \frac{\sqrt{p}}\lambda, \, \gcd(a,b) = 1\Big\},    
\end{equation}
so that
\begin{equation}\label{Neq1} N(p,n) = \sum_{(a,b) \in T} \# \{ (x,y) \in T \colon ax+by = p\}.\end{equation}
(Note that $T=\emptyset$ if $\lambda<\frac1{\sqrt2}$, a fact that reflects the observation that the sets of positive and negative directions over $\mathbb{F}_p$ do not intersect when $n<\sqrt{\frac p2}$, which we saw implicitly in Lemma~\ref{toosmall}.)
This formula reduces the estimation of $N(p,n)$ to counting solutions to Diophantine linear equations, which is an elementary task once the appropriate number-theoretic tools are in place. In Proposition~\ref{breakdown} we express this counting function using sums of basic arithmetic quantities that will be amenable to further analysis. The following notation is helpful in our discussion.

\begin{defn}
For integers $m$ and $x$ with $m \geq 2$ and $\gcd(x,m) = 1$, let $\overline{x}_m$ denote the integer in the interval $[1,m-1]$ that is the multiplicative inverse of~$x$ modulo~$m$.
\end{defn}

\begin{prop}\label{breakdown} Let $n$ be a positive integer with $\sqrt{\frac p2}<n<\sqrt{p}$, and let $\lambda$ and $T$ be as in equations~\eqref{lambda} and~\eqref{T}. Then
\[ N(p,n) =  \lambda\sqrt{p} \sum_{(a,b) \in T} \left( \frac{1}{a} + \frac{1}{b} \right) - p\sum_{(a,b) \in T} \frac{1}{ab} -\sum_{(a,b) \in T} \bigg(\bigg\{\frac{\lambda \sqrt{p}}{b}-\frac{p\overline{a}_b}{b}\bigg\} + \bigg\{\frac{\lambda\sqrt{p}}{a} -\frac{p\overline{b}_a}{a}\bigg\}-1 \bigg) .\]
\end{prop}

\begin{proof} Let $(a,b) \in T$ be arbitrary. Fix $(x_0,y_0) \in \mathbb{Z}^2$ such that $ax_0+by_0 = p$; such a solution is guaranteed to exist since $\gcd(a,b) = 1$, and moreover the set of integer solutions to $ax+by = p$ can be parameterized as $(x,y) =  (x_0+bt,y_0-at)$ for $t \in \mathbb{Z}$. Note that for any solution $(x,y) \in T$ to $ax+by = p$,
\[ \lambda \sqrt{p} \geq x = \frac{p-by}{a} \geq \frac{p -b \lambda\sqrt{p}}{a} \quad \text{and} \quad \lambda \sqrt{p} \geq y = \frac{p-ax}{b} \geq \frac{p -a \lambda\sqrt{p}}{b}.\]
Consequently, since $(x,y) =  (x_0+bt,y_0-at)$, the solutions $(x,y) \in T$ are characterized by either of the following (equivalent) inequalities: 
\begin{equation}\label{trange}
\frac{\lambda\sqrt{p}-x_0}{b} \ge t \ge \frac{p-b\lambda\sqrt{p}-ax_0}{ab} \quad \text{and} \quad \frac{y_0 - \lambda\sqrt{p}}{a} \le t \le \frac{p-a\lambda\sqrt{p}-by_0}{ab}.
\end{equation}
For any pair of real numbers $r \leq s$, the number of integers in the interval $[r,s]$ is precisely 
\[ \lfloor s \rfloor - \lceil r \rceil + 1 = (s-r+1) - ( \{s\} + \{-r\}),\]
where $\{r\} = r-\lfloor r \rfloor$ denotes the fractional part of $r$. Using this formula in equation~\eqref{trange}, with $r=(y_0 - \lambda\sqrt{p})/a$ and $s=(\lambda\sqrt{p}-x_0)/b$, we see that the number of solutions $(x,y) \in T$ to $ax+by = p$ is precisely 
\begin{equation}\label{nump} \lambda\sqrt{p}\left(\frac{1}{a} + \frac{1}{b} \right) - \frac{p}{ab}  - \left(\left\{ \frac{\lambda\sqrt{p}-x_0}{b}\right\} + \left\{\frac{ \lambda\sqrt{p}-y_0}{a} \right\} - 1 \right).\end{equation}
Since $ax_0+by_0 = p$, we have $x_0 \equiv p\overline{a}_b \pmod b$ and $y_0 \equiv p\overline{b}_a \pmod a$. As $\{ \frac cm \} = \{ \frac dm \}$ when $c\equiv d\pmod m$, we can make the substitution
\[ \left\{ \frac{\lambda\sqrt{p}-x_0}{b}\right\} + \left\{\frac{ \lambda\sqrt{p}-y_0}{a} \right\} = \left\{ \frac{\lambda\sqrt{p}-p\overline{a}_b}{b}\right\} + \left\{\frac{ \lambda\sqrt{p}-p\overline{b}_a}{a} \right\}\]
in equation~\eqref{nump}. In view of equation~\eqref{Neq1}, summing over all $(a,b) \in T$ establishes the proposition.
\end{proof}

The proof just given also allows us to show that Theorem~\ref{maint} implies Corollary~\ref{additive}; the key observation is that the expression~\eqref{nump} actually must equal either~$0$ or~$1$.

\begin{proof}[Proof of Corollary \ref{additive}]
Let $\lambda \in (\frac1{\sqrt{2}},1)$. Note that the set $T(\lambda,p)$ defined in equation~\eqref{T} is a subset of the set $\mathcal{R}_p$ defined in equation~\eqref{Rm}. In equation~\eqref{nump}, we showed that for each $(a,b) \in T(\lambda,p)$, the number of solutions $(x,y) \in T(\lambda,p)$ to the equation $ax+by=p$ is
\begin{equation}\label{nump2}
\frac{\lambda\sqrt{p}}{b}+\frac{ \lambda\sqrt{p}}{a} - \frac{p}{ab} - \bigg( \left\{ \frac{\lambda\sqrt{p}-x_0}{b}\right\} + \left\{\frac{ \lambda\sqrt{p}-y_0}{a} \right\} \bigg) + 1,
\end{equation}
where $ax_0+by_0=p$. However,
\[
\left\{ \frac{\lambda\sqrt{p}-x_0}{b} \right\} + \left\{ \frac{ \lambda\sqrt{p}-y_0}{a} \right\} \geq \left\{ \frac{\lambda\sqrt{p}-x_0}{b} + \frac{ \lambda\sqrt{p}-y_0}{a} \right\} \nonumber = \left\{ \frac{\lambda\sqrt{p}}{b}+\frac{ \lambda\sqrt{p}}{a} - \frac p{ab} \right\},
\]
which combined with equation~\eqref{nump2} implies that the number of solutions is
\begin{equation}\label{nump3}
\le \bigg\lfloor \frac{\lambda\sqrt{p}}{b}+\frac{ \lambda\sqrt{p}}{a} - \frac{p}{ab} \bigg\rfloor + 1.
\end{equation}
Moreover, the inequality $0 \leq (a-\lambda\sqrt{p})(b-\lambda\sqrt{p}) = ab+\lambda^2p-(a+b)\lambda \sqrt{p}$ implies that $\lambda a\sqrt p+\lambda b\sqrt p - p < \lambda a\sqrt p+\lambda b\sqrt p - \lambda^2 p \le ab$ since $\lambda<1$, which means that the expression inside the floor function in equation~\eqref{nump3} is less than~$1$. In other words, for each $(a,b) \in T(\lambda,p)$, there is at most one solution $(x,y) \in T(\lambda,p)$ to the equation $ax+by=p$.

Consequently, Theorem~\ref{maint} implies that there are 
\[
\bigg(\frac{12}{\pi^2}\lambda^2-D(\lambda)\bigg)p+O\big(p^{3/4}(\log p)^{2^{3/2}+1}\big)
\]
ordered pairs $(a,b) \in \Z^2 \cap [1, \sqrt{p}]^2$ for which $ax+by \equiv 0 \Mod p$ has a solution $\left(x_0, y_0\right) \in T(\lambda,p) \subset \mathcal{R}_{p}$ with $\gcd(x_0,y_0)=1$. We may let $\lambda \to 1^-$ since the implicit constant is absolute, giving a main term of $\big( \frac{12}{\pi^2}-1 \big) p$ by continuity; the corollary now follows from Theorem~\ref{AC}.
\end{proof}

The only task that remains is to prove Theorem~\ref{maint}; we do so by using Proposition~\ref{breakdown} to divide the proof into two subtasks. In Section~\ref{secpropA} we estimate 
\begin{equation} \label{subtask1}
\lambda\sqrt{p} \sum_{(a,b) \in T} \left( \frac{1}{a} + \frac{1}{b} \right) - p\sum_{(a,b) \in T} \frac{1}{ab},
\end{equation}
from which the main term of $N(p,n)$ arises In Section~\ref{secpropB} we estimate 
\begin{equation} \label{subtask2}
\sum_{(a,b) \in T} \bigg(\bigg\{\frac{\lambda \sqrt{p}}{b}-\frac{p\overline{a}_b}{b}\bigg\} + \bigg\{\frac{\lambda\sqrt{p}}{a} -\frac{p\overline{b}_a}{a}\bigg\}-1\bigg),
\end{equation}
which contributes only to the error term of $N(p,n)$. In particular, Theorem~\ref{maint} follows immediately from combining Propositions~\ref{breakdown}, \ref{propa}, and~\ref{propb}.

\section{Main term estimation}\label{secpropA}

The goal of this section is to establish Proposition~\ref{propa}, giving an asymptotic formula for the expression~\eqref{subtask1} and thus eventually for $N(p,n)$. The key idea behind the estimates in this section is that a double sum over lattice points in a region can be approximated by a suitable double integral, and that the contribution to the sum from visible lattice points can then be isolated using the M\"obius function. First we establish by elementary means a bound for the difference between the double sum and the corresponding integral.

\begin{lemma} \label{lemsum2int}
Let $f \colon \mathbb{R}_{>0}^2 \to \mathbb{R}_{>0}$ be a positive function that is decreasing in both arguments. Let~$r$ and~$s$ be positive integers with $r < s$, and let~$r'$ and~$s'$ be real numbers satisfying $r \leq r' < r+1$ and $s-1<s'\leq s$ and $r'<s' < 2r'$. Then
\begin{equation}\label{sum2int}
\sum_{\ell=s-r}^r \sum_{k = s-\ell}^r f(k,\ell) =\int_{s'-r'}^{r'}\int_{s'-y}^{r'} f(x,y) \,dx \,dy + \epsilon
\end{equation}
where 
\begin{equation}\label{epsilon}
|\epsilon| \leq \int_{s'-r'}^{r'}\int_{s'-y}^{s'-y+5} f(x,y) \,dx \,dy + 2 \sum_{\ell = s-r}^r f(s-\ell,\ell).
\end{equation}
\end{lemma}

\begin{proof}
First we remark that if $s \geq 2r-1$, it is easy to check that the double sum in equation~\eqref{sum2int} (which might even be an empty sum) is bounded by the sum in equation~\eqref{epsilon} and that the double integral in equation~\eqref{sum2int} is bounded by the double integral in equation~\eqref{epsilon}. Therefore we may assume that $s\leq 2r-2$. Decompose the sum 
\begin{multline}\label{sumdecomp}
\sum_{\ell=s-r}^r \sum_{k = s-\ell}^r f(k,\ell) = \sum_{\ell=s-r+2}^r \sum_{k = s-\ell+2}^r f(k,\ell) \\
+ \bigg( \sum_{\ell=s-r}^r \big( f(s-\ell,\ell) + f(s-\ell+1,\ell) \big) - f(r+1,s-r) \bigg).
\end{multline}
The parenthetical expression is bounded above by the second term on the right-hand side of equation~\eqref{epsilon}; thus it suffices to show that the difference between the double sum on the right-hand side of equation~\eqref{sumdecomp} and the double integral in equation~\eqref{sum2int} is bounded above in absolute value by the double integral in equation~\eqref{epsilon}.

The fact that $f(x,y)$ is decreasing in both arguments implies the inequalities 
\begin{equation} \label{debug1}
\int_{\ell}^{\ell+1} \int_{k}^{k+1} f(x,y) \,dx \,dy \leq f(k,\ell) \leq \int_{\ell-1}^{\ell} \int_{k-1}^{k} f(x,y) \,dx \,dy.
\end{equation}
Summing the first inequality over~$k$ and~$\ell$ yields
\begin{align*}
\sum_{\ell=s-r+2}^r \sum_{k = s-\ell+2}^r f(k,\ell) &\ge \int_{s-r+2}^{r+1}\int_{s-\lfloor y\rfloor+2}^{r+1} f(x,y) \,dx \,dy \\
&\ge \int_{s-r+3}^{r+1}\int_{s-y+3}^{r+1} f(x,y) \,dx \,dy \ge \int_{s'-r'+5}^{r'}\int_{s'-y+5}^{r'} f(x,y) \,dx \,dy
\end{align*}
by the positivity of~$f$ and the assumptions on~$r'$ and~$s'$. Similarly, summing the second inequality of equation~\eqref{debug1} over~$k$ and~$\ell$ yields
\begin{align*}
\sum_{\ell=s-r+2}^r \sum_{k = s-\ell+2}^r f(k,\ell) &\leq \int_{s-r+1}^{r}\int_{s-\lfloor y\rfloor+1}^{r} f(x,y) \,dx \,dy  \\
&\leq \int_{s-r+1}^{r}\int_{s-y+1}^{r} f(x,y) \,dx \,dy \leq \int_{s'-r'}^{r'}\int_{s'-y}^{r'} f(x,y) \,dx \,dy .
\end{align*}
From these two chains of inequalities, we see that the double sum on the right-hand side of equation~\eqref{sumdecomp} is smaller than the double integral in equation~\eqref{sum2int}, but by no more than
\[
\int_{s'-r'}^{r'}\int_{s'-y}^{r'} f(x,y) \,dx \,dy - \int_{s'-r'+5}^{r'}\int_{s'-y+5}^{r'} f(x,y) \,dx \,dy \le \int_{s'-r'}^{r'}\int_{s'-y}^{s'-y+5} f(x,y) \,dx \,dy,
\]
recovering the double integral in equation~\eqref{epsilon} and thus establishing the lemma. 
\end{proof}

We quickly evaluate two double integrals that will arise when applying this lemma.

\begin{lemma} \label{iints}
For any real numbers $\alpha$ and $\beta$ satisfying $0<\beta<\alpha<2\beta$,
\begin{align*}
\int_{\alpha-\beta}^{\beta} \int_{\alpha-y}^\beta \bigg( \frac1x + \frac1y \bigg) \ dx\, dy &= 2 \bigg( 2\beta-\alpha + (\alpha-\beta) \log\frac{\alpha-\beta}\beta \bigg), \\
\int_{\alpha-\beta}^{\beta} \int_{\alpha-y}^\beta \frac{1}{xy} \ dx\, dy &= 2\Li_2 \bigg( \frac\beta\alpha \bigg) + \log^2 \frac\alpha\beta - \frac{\pi^2}{6},
\end{align*}
where the dilogarithm function $\Li_2$ was defined in equation~\eqref{Li}.
\end{lemma}

\begin{proof}
The first double integral is straightforward to evaluate. For the second double integral, straightforward methods using the definition~\eqref{Li} yield
\[
\int_{\alpha-\beta}^{\beta} \int_{\alpha-y}^\beta \frac{1}{xy} \ dx\, dy = \Li_2\bigg( \frac\beta\alpha \bigg) - \Li_2\bigg( 1 - \frac\beta\alpha \bigg) + \log\frac\alpha\beta \cdot \log\frac{\alpha-\beta}\beta,
\]
which can be transformed into the desired form using the well-known functional equation of the dilogarithm,
\[
\Li_2(z) + \Li_2(1-z) = \frac{\pi^2}{6}-\log z \cdot \log (1-z)
\]
(see for example~\cite[Section 2]{Loxton}), valid for $0<z<1$.
\end{proof}

Using Lemma~\ref{lemsum2int}, we now find an asymptotic formula for the two sums in equation~\eqref{subtask1} (forming the main term in Theorem~\ref{maint}), using the M\"obius function to detect visible lattice points as was done in the proof of Lemma~\ref{Qcount}.

\begin{prop}\label{propa}
Let~$n$ be a positive integer satisfying $\sqrt{\frac p2}<n<\sqrt{p}$, and let~$\lambda$ and~$T$ be as defined in equations~\eqref{lambda} and~\eqref{T}. Then
\begin{equation}\label{propa two sums}
\lambda\sqrt{p} \sum_{(a,b) \in T} \left( \frac{1}{a} + \frac{1}{b} \right) - p\sum_{(a,b) \in T} \frac{1}{ab} = \bigg(\frac{12}{\pi^2} \lambda^2 -D(\lambda)\bigg) p + O\left( \sqrt{p}\log^2 p \right).
\end{equation}
\end{prop}

\begin{proof}
For the first sum, the characteristic property~\eqref{key identity} of the M\"obius function gives us
\begin{equation} \label{first mobius}
\sum_{(a,b) \in T} \left( \frac{1}{a} + \frac{1}{b} \right) = \sum_{d = 1}^n \mu(d) \sum_{\substack{(a,b) \in [n]^2 \\  a+b \geq \frac{\sqrt p}\lambda \\ d\mid a,\, d\mid b}} \left( \frac{1}{a} + \frac{1}{b} \right)  = \sum_{d = 1}^n \frac{\mu(d)}{d} \sum_{\ell = \lceil \frac{\sqrt{p}}{\lambda d}\rceil - \lfloor \frac{n}{d} \rfloor}^{\lfloor \frac{n}{d} \rfloor} \sum_{k = \lceil \frac{\sqrt{p}}{\lambda d}\rceil -\ell }^{\lfloor \frac{n}{d} \rfloor} \left( \frac{1}{k} + \frac{1}{\ell} \right)
\end{equation}
upon setting $a=\ell d$ and $b=k d$.
We apply Lemma~\ref{lemsum2int} to this inner double sum, with 
\begin{equation}\label{rsr's'}
r = \bigg\lfloor \frac{n}{d} \bigg\rfloor, \quad s = \bigg\lceil \frac{\sqrt{p}}{\lambda d} \bigg\rceil, \quad r' = \frac{n}{d}=\frac{\lambda \sqrt{p}}{d}, \quad \text{and }s' = \frac{\sqrt{p}}{\lambda d}
\end{equation}
(note that $r'<s'<2r'$ since $\frac12<\lambda^2<1$).
Using Lemma~\ref{iints}, the double integral in equation~\eqref{sum2int} becomes after simplification
\begin{align*}
\int_{\frac{\sqrt{p}}{\lambda d} - \frac{n}{d} }^{\frac{n}{d}} \int_{\frac{\sqrt{p}}{\lambda d} -y }^{\frac{n}{d}} \bigg(\frac{1}{x} + \frac{1}{y}\bigg) \,dx \,dy = \frac{2\sqrt p}d \big( (\lambda^{-1}-\lambda) \log(\lambda^{-2}-1) + 2\lambda-\lambda^{-1} \big)
\end{align*}
(remembering that $n=\lambda\sqrt p$).
On the other hand, the $\epsilon$ in equation~\eqref{sum2int} is at most 
\begin{align*}
\int_{s'-r'}^{r'} & \int_{s'-y}^{s'-y+5} \bigg(\frac{1}{x} + \frac{1}{y}\bigg) \,dx \,dy + 2 \sum_{\ell = s-r}^r \bigg(\frac{1}{s-\ell} + \frac{1}{\ell}\bigg) \\
& = \int_{s'-r'}^{r'}\left( \log \left( 1+\frac{5}{s'-y} \right) + \frac{5}{y} \right) \,dy + O\bigg( \sum_{\ell = s-r}^r \frac{1}{\ell} \bigg) \\
&\ll \int_{s'-r'}^{r'}\left( \frac{1}{s'-y}+\frac{1}{y} \right) \ dy + \log \left( \frac{r}{s-r} \right) \ll \log \left( \frac{r}{s-r} \right) = \log \frac{\lambda^2}{1-\lambda^2}
\end{align*}
by equation~\eqref{harmonic} and the bound $\log(1+x) \leq x$.
Substituting these two evaluations back into equation~\eqref{first mobius} and multiplying by $\lambda\sqrt p$ gives 
\begin{align}
\lambda\sqrt p & \sum_{(a,b) \in T} \left( \frac{1}{a} + \frac{1}{b} \right) \nonumber \\
&= \lambda\sqrt p \sum_{d = 1}^n \frac{\mu(d)}{d} \left( \frac{2\sqrt p}d \big( (\lambda^{-1}-\lambda) \log(\lambda^{-2}-1) + 2\lambda-\lambda^{-1} \big) + O \bigg( \log \frac{\lambda^2}{1-\lambda^2} \bigg) \right) \nonumber \\
&= \frac{12p}{\pi^2}\left((1-\lambda^2 )\log(\lambda^{-2}-1) + (2\lambda^2-1) \right) + O\left(\sqrt{p}\log p \cdot \log \frac{\lambda^2}{1-\lambda^2} \right).
\label{firstestimate}
\end{align}
using equations~\eqref{musum2} and~\eqref{harmonic} and $\lambda<1$.

For the second sum on the left-hand side of equation~\eqref{propa two sums}, the same procedure yields
\begin{equation} \label{second mobius}
\sum_{(a,b) \in T} \frac{1}{ab}  =  \sum_{d = 1}^n \frac{\mu(d)}{d^2} \sum_{\ell = \lceil \frac{\sqrt{p}}{\lambda d}\rceil - \lfloor \frac{n}{d} \rfloor}^{\lfloor \frac{n}{d} \rfloor} \sum_{k = \lceil \frac{\sqrt{p}}{\lambda d}\rceil -\ell }^{\lfloor \frac{n}{d} \rfloor} \frac{1}{k\ell}.
\end{equation}
Once again we apply Lemma~\ref{lemsum2int} using the parameters from equation~\eqref{rsr's'}.
Using Lemma~\ref{iints}, the double integral in equation~\eqref{sum2int} becomes after simplification
\begin{align*}
\int_{\frac{\sqrt{p}}{\lambda d} - \frac{n}{d} }^{\frac{n}{d}} \int_{\frac{\sqrt{p}}{\lambda d} -y }^{\frac{n}{d}} \frac{1}{xy} \,dx \,dy = 2\Li_2(\lambda^2) + \log^2(\lambda^2) - \frac{\pi^2}{6}.
\end{align*}
On the other hand, the~$\epsilon$ in equation~\eqref{sum2int} is similarly bounded by
\begin{align*}
\int_{s'-r'}^{r'} & \int_{s'-y}^{s'-y+5} \frac{1}{xy}  \,dx \,dy + 2 \sum_{\ell = s-r}^r \frac{1}{(s-\ell)\ell} \\
&= \int_{s'-r'}^{r'} \frac{1}{y}\log\left( 1 + \frac{5}{s'-y} \right) \, dy + \frac 2s \sum_{\ell = s-r}^r \bigg( \frac1\ell + \frac1{s-\ell} \bigg) \\
& \ll \int_{s'-r'}^{r'} \frac{1}{y(s'-y)} \,dy + \frac{1}{s}\log \left( \frac{r}{s-r} \right) \\
& \ll \frac{1}{s'}\log \left( \frac{r'}{s'-r'} \right) = \frac{d}{\sqrt{p}}\log \frac{\lambda^2}{1-\lambda^2}.
\end{align*}
Substituting these two evaluations back into equation~\eqref{second mobius} and multiplying by~$p$ gives
\begin{align*}
p \sum_{(a,b) \in T} \frac{1}{ab} &= p \sum_{d = 1}^n \frac{\mu(d)}{d^2} \bigg( 2\Li_2(\lambda^2) + \log^2(\lambda^2) - \frac{\pi^2}{6} + O\bigg( \frac{d}{\sqrt{p}}\log \frac{\lambda^2}{1-\lambda^2} \bigg) \bigg) \\
&= \frac{6p}{\pi^2} \big( 2\Li_2(\lambda^2) + \log^2(\lambda^2) \big) - p + O\bigg( \sqrt p\log p \cdot \log \frac{\lambda^2}{1-\lambda^2} \bigg)
\end{align*}
using equations~\eqref{musum2} and~\eqref{harmonic} and the fact that $\Li_2(\lambda^2)$ is bounded for $\frac1{\sqrt2}<\lambda<1$.

Finally, subtracting this equation from equation~\eqref{firstestimate} yields
\begin{align*}
\lambda\sqrt{p} \sum_{(a,b) \in T} & \left( \frac{1}{a} + \frac{1}{b} \right) - p\sum_{(a,b) \in T} \frac{1}{ab} \\
&= \frac{12p}{\pi^2}\left((1-\lambda^2 )\log(\lambda^{-2}-1) + (2\lambda^2-1) \right) \\
&\qquad{}- \frac{6p}{\pi^2} \big( 2\Li_2(\lambda^2) + \log^2(\lambda^2) \big) + p + O\left(\sqrt{p}\log p \cdot \log \frac{\lambda^2}{1-\lambda^2} \right) \\
&= \bigg(\frac{12}{\pi^2} \lambda^2 -D(\lambda)\bigg) p + O\left(\sqrt{p}\log p \cdot \log \frac{\lambda^2}{1-\lambda^2} \right)
\end{align*}
by the definition~\eqref{D def} of $D(\lambda)$ in this range. Note that $\lambda^2p = n^2 \leq p-1$ by assumption, and so $p \ge (1-\lambda^2)^{-1}$ and hence $-\log(1-\lambda^2) \leq \log p$. Consequently, we can replace the error term in this last estimate with $\sqrt{p} \log^2p$, which concludes the proof of the proposition.
\end{proof}

\section{Error term estimation}\label{secpropB}

The final goal of this paper is to establish an estimate for the expression~\eqref{subtask2} that allows it to be absorbed into the error term in Theorem~\ref{maint}. Indeed, since $(a,b)\in T$ if and only if $(b,a)\in T$, it suffices to estimate
\begin{equation} \label{summand}
\sum_{(a,b) \in T} \bigg(\bigg\{\frac{\lambda \sqrt{p}}{b}-\frac{p\overline{a}_b}{b}\bigg\} - \frac12 \bigg),
\end{equation}
which we do in Proposition~\ref{propb}. (As mentioned earlier, Theorem~\ref{maint} follows immediately from combining Propositions~\ref{breakdown}, \ref{propa}, and~\ref{propb}.) Intuitively, one expects the average value of the summand in equation~\eqref{summand} to be close to~$0$, since the argument of the fractional-part function seems randomly distributed; this intuition can be made precise by bounding the discrepancy of the summand (see Definition~\ref{discrepancy} below).

In the following discussion, we fix an odd prime~$p$. We use the standard notations $\tau(n)$ for the number of positive divisors of~$n$ and $\phi(n)$ for the number of integers in~$[n]$ that are coprime to~$n$; and we recall that~$T$ was defined in equation~\eqref{T}.

\begin{defn} \label{intervals def}
For each $b \in [(\lambda^{-1}-\lambda) \sqrt{p},\lambda\sqrt{p}]$ define the following finite sets:
\begin{align*}
I_b &= \bigg\{ a \in \mathbb{Z} \colon (a,b) \in T \bigg\} = \bigg\{a \in \Big[ \frac{\sqrt p}\lambda - b,  \lambda\sqrt p \Big] \colon \gcd(a,b) = 1 \bigg\}, \\
I_b^+ &= \bigg\{a \in \Big[ \frac{\sqrt p}\lambda - b, \frac{\sqrt p}\lambda \Big) \colon \gcd(a,b) = 1 \bigg\}, \\
I_b^- &= \bigg\{a \in \Big( \lambda \sqrt{p},  \frac{\sqrt p}\lambda \Big) \colon \gcd(a,b) = 1 \bigg\}.
\end{align*}
\end{defn}

Clearly $I_b = I_b^+\setminus I_b^-$, and the fact that $T=\big\{ (a,b)\colon b\in [(\lambda^{-1}-\lambda) \sqrt{p},\lambda\sqrt{p}],\, a\in I_b\big\}$ follows directly from the definition~\eqref{T}. The set~$I_b^-$ is contained in an interval whose length is independent of~$b$, and we will estimate the contribution to equation~\eqref{summand} from $a\in I_b^-$ using exponential sums. On the other hand, the set~$I_b^+$ is a complete set of reduced residues modulo~$b$, allowing the contribution to equation~\eqref{summand} from $a\in I_b^+$ to be estimated using elementary techniques as follows.
The next lemma, which uses a classical Bernoulli polynomial identity, is all we need to estimate \eqref{summand} over the interval $I_b^+$. 

\begin{lemma}\label{I+}
For any real numbers $\alpha$ and~$y$ and any positive integer~$b$,
\[
\sum_{\substack{y \leq a < y+b \\ \gcd(a,b) = 1}} \left(\left\{ \alpha - \frac{a}{b} \right\}-\frac{1}{2} \right) \ll \tau(b) .
\]
\end{lemma}

\begin{proof}
Since $\{ \alpha - \frac{x}{b} \}$ is periodic with period~$b$, it suffices to consider $y=1$.
For all positive integers~$q$, we have the identity
\begin{equation} \label{Bern}
\sum_{k = 1}^q \left(\left\{ \alpha - \frac{k}{q} \right\} - \frac{1}{2} \right)  =  \{\alpha q\} - \frac{1}{2}
\end{equation}
(see for example~\cite[Lemma 2]{Mon}). Using the property~\eqref{key identity} of the M\"obius function, we write
\begin{align*}
    \sum_{\substack{1 \leq a \leq b \\ \gcd(a,b) = 1}} \left(\left\{ \alpha - \frac{a}{b} \right\} - \frac{1}{2} \right) &= \sum_{1 \leq a \leq b} \left(\left\{ \alpha - \frac{a}{b} \right\} - \frac{1}{2} \right) \sum_{d\mid \gcd(a,b)} \mu(d) \\
    &= \sum_{d\mid b} \mu(d) \sum_{\substack{1 \le a \le b \\ d\mid a}} \left(\left\{ \alpha - \frac{a}{b} \right\} - \frac{1}{2} \right)\\
    &= \sum_{d\mid b} \mu(d) \sum_{k = 1}^{b/d}\left(\left\{ \alpha - \frac{kd}{b} \right\} - \frac{1}{2} \right)\\
    &= \sum_{d\mid b} \mu(d) \bigg(\bigg\{\frac{\alpha b}{d}\bigg\} - \frac{1}{2}\bigg) \ll \tau(b)
\end{align*}
by equation~\eqref{Bern} applied with $q=\frac bd$.
\end{proof}

To estimate the contribution to equation~\eqref{summand} from intervals of the shape~$I_b^-$, we consider the discrepancy of the corresponding sequence.

\begin{defn}\label{discrepancy}
Let $\{u_n\}$ be a sequence. For all $0 \leq \alpha \leq \beta \leq 1$, define 
\[
Z(N;\alpha, \beta)=\#\big\{ n \in [N]\colon u_n \in [\alpha, \beta] \Mod 1 \big\}.
\]
The \emph{discrepancy} of the sequence $\{u_n\}$,
\[
D(N)=\sup_{0 \leq \alpha \leq \beta \leq 1} \big| Z(N; \alpha, \beta)-N(\beta-\alpha) \big|.
\]
measures the maximum absolute difference between the counting function $Z(N; \alpha, \beta)$ and the expected number $N(\beta-\alpha)$.
\end{defn}

In the following well-known inequality (see for example~\cite[Corollary 1.1]{Ten}), we use the standard notation $e(x) = e^{2\pi i x}$. 

\begin{prop}[Erd\H{o}s--Tur\'an inequality]\label{ET}
For any sequence $\{u_n\}$ and any positive integers~$N$ and~$K$,
\[
D(N) \leq \frac{N}{K+1} + 3 \sum_{t=1}^K \frac{1}{t} \bigg| \sum_{n=1}^N e(tu_n)\bigg|.
\]
\end{prop}

Our application of Proposition~\ref{ET} will use an estimate on incomplete Kloosterman sums which ultimately follows from Weil's bounds on exponential sums.

\begin{lemma} \label{expsum}
Let $m\ge2$ be an integer and $y$ and $z$ real numbers satisfying $0<z-y\ll m$. Then for any integer~$t$,
\[
\sum_{\substack{{y<n\le z}\\{(n,m)=1}}} e\Big ( \frac{t\overline n_m}{m}\Big ) \ll \sqrt{m\gcd(t,m)}\cdot \tau (m)\log m.
\]
\end{lemma}

\begin{proof}
Dartyge and the first author~\cite[Lemma 1]{DM} showed that for arbitrary real numbers $y<z$,
\[
\sum_{\substack{{y<n\le z}\\{(n,m)=1}}} e\Big ( \frac{t\overline n_m}{m}\Big )= \frac{z-y}{m} \mu\Big (\frac{m}{\gcd(t,m)}\Big )\frac{\phi (m)}{\phi (m/\gcd(t,m))}+O\big( \sqrt{m\gcd(t,m)}\cdot \tau (m)\log m \big).
\]
(While estimates for incomplete Kloosterman sums have been recorded for decades, this more recent citation has the desirable properties that a complete proof is included and that the error term does not contain an $m^\epsilon$ factor.) Using the elementary inequality $\phi(mn) \le m\phi(n)$, so that $\phi(c)/\phi(\frac cd)\le d$ when $d\mid c$, we find that
\[
\frac{\phi (m)}{\phi (m/\gcd(t,m))} \le \gcd(t,m) \le \sqrt{m\gcd(t,m)};
\]
therefore the first term can be subsumed into the error term in light of the assumption $z-y\ll m$.
\end{proof}

The below lemma combines the previous two results to give a sufficient estimate of \eqref{summand} over $I_b^-$. For a positive integer~$b$, we define the counting function of the totients modulo~$b$,
\[ R_b(X) = \{ a \in [1,X] \colon \gcd(a,b) = 1\} .\]

\begin{lemma}\label{estfrac} Fix a positive integer $b<\sqrt{p}$. For any positive real $X$, define
\begin{align*}
Z_b(X,\alpha,\beta) &= \# \Big\{ a \in R_b(X) \colon \frac{p \overline{a}_b}b \in [\alpha,\beta] \Mod{1}\Big\}, \\
D_b(X) &=\sup_{0 \leq \alpha \leq \beta \leq 1} \big|Z_b(X,\alpha,\beta)-(\beta-\alpha)|R_b(X)|\big|.
\end{align*}
Then for all real numbers $0\le X \le b$,
\[ D_b(X) \ll \tau(b)^{3/2} p^{1/4} (\log p)^2.\]
\end{lemma}

\begin{proof}
For each integer~$t$, Lemma~\ref{expsum} gives the estimate
\begin{align*} \sum_{a \in R_b(X)} e\left(\frac{tp\overline{a}_b}{b} \right) & \ll \sqrt{b \gcd(pt,b)}\cdot \tau(b) \log b \ll \sqrt{b \gcd(t,b)} \cdot \tau(b) \log b
\end{align*}
(since $\gcd(p,b)=1$). For any positive integer~$K$, using the change of variables $t=ds$,
\begin{multline*}
\sum_{t \leq K} \frac{\sqrt{\gcd(t,b)}}{t} = \sum_{d\mid b} \sqrt{d} \sum_{\substack{t \leq K \\ \gcd(t,b) = d}} \frac1t
\le \sum_{d\mid b} \sqrt{d} \sum_{\substack{t \leq K \\ d\mid t}} \frac{1}{t} = \sum_{d\mid b} \frac1{\sqrt d} \sum_{s\le K/d} \frac1s \\
\le \sum_{d\mid b} \frac1{\sqrt d} \sum_{s\le K} \frac1s
\ll  \sum_{d\mid b} \frac{1}{\sqrt{d}} \log K \leq \sum_{d = 1}^{\tau(b)}\frac{1}{\sqrt{d}} \log K \ll \sqrt{\tau(b)}\log K.
\end{multline*}
Applying Proposition~\ref{ET} and using these two estimates, we obtain
\begin{align*}
    D_b(X) &\leq \frac{|R_b(X)|}{K+1} +3\sum_{t \leq K} \frac{1}{t} \left| \sum_{a \in R_b(X)} e\left(\frac{tp\overline{a}_b}{b} \right) \right| \\
    &\ll \frac{X}K + \sum_{t \leq K} \frac{1}{t} \sqrt{b \gcd(t,b)} \cdot \tau(b) \log b \\
    & \ll \frac{b}K+ \big( \sqrt{b} \cdot \tau(b) \log b \big) \tau(b) \log K \ll \frac{b}K+ \tau(b)^{3/2} p^{1/4}\log p \log K,
\end{align*}
since $b<\sqrt p$.
Setting $K = b$ completes the proof of the lemma.
\end{proof}

We remark that Karatsuba~\cite{Ka1, Ka2} gave estimates on equidistribution of fractional parts in much shorter intervals. However, the error terms in those estimates are less advantageous for us, so we have opted for this more elementary method.

We are now ready to prove Proposition~\ref{propb}, which completes the proof of Theorem~\ref{maint}.

\begin{prop} \label{propb}
Let $\lambda \in (\frac1{\sqrt{2}},1)$. Then with $T$ as defined in equation~\eqref{T},
\[
\sum_{(a,b) \in T} \bigg(\bigg\{\frac{\lambda \sqrt{p}}{b}-\frac{p\overline{a}_b}{b}\bigg\} - \frac12 \bigg) \ll p^{3/4}(\log p)^{2^{3/2}+1}.
\]
\end{prop}

\begin{proof}
We begin by noting that equation~\eqref{T} and Definition~\ref{intervals def} imply
\begin{align}
\sum_{(a,b) \in T} \bigg(\bigg\{\frac{\lambda \sqrt{p}}{b}-\frac{p\overline{a}_b}{b}\bigg\} - \frac12 \bigg) &= \sum_{(\lambda^{-1}-\lambda)\sqrt{p} \le b \le \lambda\sqrt p} \sum_{a\in I_b} \bigg(\bigg\{\frac{\lambda \sqrt{p}}{b}-\frac{p\overline{a}_b}{b}\bigg\} - \frac12 \bigg) \nonumber \\
&= \sum_{(\lambda^{-1}-\lambda)\sqrt{p} \le b \le \lambda\sqrt p} \sum_{a\in I_b^+} \bigg(\bigg\{\frac{\lambda \sqrt{p}}{b}-\frac{p\overline{a}_b}{b}\bigg\} - \frac12 \bigg) \label{decomp} \\
&\qquad{}- \sum_{(\lambda^{-1}-\lambda)\sqrt{p} \le b \le \lambda\sqrt p} \sum_{a\in I_b^-} \bigg(\bigg\{\frac{\lambda \sqrt{p}}{b}-\frac{p\overline{a}_b}{b}\bigg\} - \frac12 \bigg). \nonumber
\end{align}
Note that $\{p\overline{a}_b \colon a \in I_b^+\}$ comprises a full set of distinct reduced residues modulo~$b$. Therefore by Lemma~\ref{I+},
\[ \sum_{a \in I_b^+} \left( \bigg\{\frac{\lambda \sqrt{p}}{b}-\frac{p\overline{a}_b}{b}\bigg\} - \frac{1}{2} \right)   \ll \tau(b), \]
from which it follows that
\begin{equation}\label{plussum}
\sum_{(\lambda^{-1}-\lambda)\sqrt{p} \le b \le \lambda\sqrt p} \sum_{a \in I_b^+} \left( \bigg\{\frac{\lambda \sqrt{p}}{b}-\frac{p\overline{a}_b}{b}\bigg\} - \frac{1}{2} \right) \ll \sum_{b \leq \sqrt{p}} \tau(b) \ll \sqrt{p}\log p.
\end{equation}

Turning to the last double sum in equation~\eqref{decomp}, we define the function $h_b \colon [0,1] \to \mathbb{R}$ by
\[ h_b(\alpha) = \#\bigg\{ a\in I_b^- \colon \bigg\{\frac{\lambda \sqrt{p}}{b}-\frac{p\overline{a}_b}{b}\bigg\} \geq \alpha \bigg\},\]
so that
\begin{equation}\label{hswitch} \sum_{a \in I_b^-}  \bigg\{\frac{\lambda \sqrt{p}}{b}-\frac{p\overline{a}_b}{b}\bigg\}  = \int_0^1 h_b(\alpha) \ d\alpha .\end{equation}
For all real numbers~$\beta$ and~$\gamma$, if $\{\beta \} \geq \{\gamma\}$ then $\{\beta - \gamma\} = \{\beta \}-\{\gamma\}$, while if $\{\beta\}<\{\gamma\}$ then $\{\beta - \gamma\} = \{\beta \}-\{\gamma\}+1$. Consequently, for all integers~$b$ and real numbers $\alpha \in [0,1)$,
\[
\bigg\{ a \in I_b^- \colon \Big\{ \beta - \frac{p\overline{a}_b}b \Big\} \geq \alpha \bigg\} = \begin{cases}
  \big\{a \in I_b^- \colon \{\beta\} < \{\frac{p\overline{a}_b}b\} \le \{\beta\} + 1 - \alpha \big\}, & \text{if } \{\beta\} \leq \alpha, \\
  \big\{a \in I_b^- \colon \{\frac{p\overline{a}_b}b\} \leq \{\beta \}-\alpha \text{ or } \{\frac{p\overline{a}_b}b\} > \{\beta\} \big\}, & \text{if } \{\beta\} > \alpha.
\end{cases}
\]
In either case, $\big\{ \beta - \frac{p\overline{a}_b}b \big\} \ge \alpha$ if and only if $\{\frac{p\overline{a}_b}b\}$ lies in an interval, or union of intervals, of total length $1-\alpha$.
It thus follows from Lemma~\ref{estfrac} that $h_b(\alpha) - (1-\alpha)|I_b^-| \ll \tau(b)^{3/2} p^{1/4} (\log p)^2$. Substituting into equation~\eqref{hswitch}, we obtain
\begin{align*} \sum_{a \in I_b^-} \left( \bigg\{\frac{\lambda \sqrt{p}}{b}-\frac{p\overline{a}_b}{b}\bigg\} - \frac{1}{2} \right) &= \int_0^1 \big( (1-\alpha)|I_b^-| + O(\tau(b)^{3/2} p^{1/4} (\log p)^2) \big) \,d\alpha - \sum_{a \in I_b^-} \frac12 \\
&= |I_b^-|\int_0^1 (1-\alpha) \ d\alpha - \frac{1}{2} |I_b^-| +O\big( \tau(b)^{3/2} p^{1/4} (\log p)^2\big) \\
& \ll \tau(b)^{3/2} p^{1/4} (\log p)^2, \end{align*}
whereupon
\begin{align*}
\sum_{(\lambda^{-1}-\lambda)\sqrt{p} \le b \le \lambda\sqrt p} \sum_{a \in I_b^-} \left( \bigg\{\frac{\lambda \sqrt{p}}{b}-\frac{p\overline{a}_b}{b}\bigg\} - \frac{1}{2} \right) &\ll  p^{1/4}(\log p)^2 \sum_{(\lambda^{-1}-\lambda)\sqrt{p} \le b \le \lambda\sqrt p}\tau(b)^{3/2} \\
&\le p^{1/4}(\log p)^2 \sum_{b\le \sqrt{p}} \tau(b)^{3/2} \\
&\ll p^{1/4}(\log p)^2 \cdot \sqrt p (\log p)^{2^{3/2}-1}
\end{align*}
from known bounds for sums of powers of $\tau(b)$ (see for example~\cite[equation~(2.31)]{MV}).
Inserting this estimate and the estimate~\eqref{plussum} into equation~\eqref{decomp} completes the proof of the proposition.
\end{proof}

\section*{Acknowledgments}
The authors thank J\'ozsef Solymosi for suggesting this project and pointing out Corollary~\ref{TVcor}, and Joshua Zahl for helpful discussions. The research of the first author was supported in part by a National Sciences and Engineering Research Council of Canada Discovery Grant. The research of the second author was supported in part by Killam and NSERC doctoral scholarships. The research of the third author was supported in part by a Four Year Doctoral Fellowship from the University of British Columbia.

\end{document}